\documentclass[reqno, 12pt]{amsart}
\pdfoutput=1
\makeatletter
\let\origsection=\section \def\section{\@ifstar{\origsection*}{\mysection}} 
\def\mysection{\@startsection{section}{1}\z@{.7\linespacing\@plus\linespacing}{.5\linespacing}{\normalfont\scshape\centering\S}}
\makeatother        
\usepackage{amsmath,amssymb,amsthm}    
\usepackage{mathabx}\changenotsign  
\usepackage{bbm, accents}  
\usepackage{mathrsfs} 
\usepackage{dsfont} 
\usepackage[babel]{microtype}

\usepackage{xcolor}  	
\usepackage[backref]{hyperref}
\hypersetup{
	colorlinks,
    linkcolor={red!60!black},
    citecolor={green!60!black},
    urlcolor={blue!60!black},
}

\usepackage{bookmark}

\usepackage[abbrev,msc-links,backrefs]{amsrefs} 

\usepackage{doi}

\renewcommand{\PrintDOI}[1]{\doi{#1}}

\usepackage[T1]{fontenc}
\usepackage{lmodern}

\usepackage[english]{babel}
\numberwithin{equation}{section}

\usepackage{geometry}
\usepackage{enumitem}

\def\rmlabel{\upshape({\itshape \roman*\,})}

\def\qqand{\qquad\text{and}\qquad}

\let\setminus=\smallsetminus

\newtheorem{theorem} {Theorem}
\newtheorem{lemma} [theorem]{Lemma}

\newtheorem{corollary} {Corollary} 
 
\newtheorem{claim} {Claim}
\theoremstyle{definition} 

\newtheorem{remark} [theorem] {Remark} 
\newtheorem{definition}[theorem] {Definition}
\newtheorem{construction}[theorem] {Construction}

\newcommand{\cA}{\mathcal{A}}

\newcommand{\cC}{\mathcal{C}}

\newcommand{\cP}{\mathcal{P}}

\newcommand{\eps}{\varepsilon}

\newcommand{\irnumber}[1]{r_{\mathrm{ind}}(#1)}
\newcommand{\nirnumber}[1]{r(#1)}

\newcommand{\PR}{\mathbf{P}}

\newcommand{\gnp}{\mathbb{G}^{(k)}(n,p)}
\newcommand{\monosize}{M}

\DeclareMathOperator{\aut}{aut}
\DeclareMathOperator{\emb}{emb}

\begin{document}
\title{A note on induced Ramsey numbers}

\author[D. Conlon]{David Conlon} \address{Mathematical Institute\\ University of Oxford\\
Oxford\\ United Kingdom} \email{david.conlon@maths.ox.ac.uk}
\thanks{The first author was supported by a Royal Society University Research Fellowship}
 
\author[D. Dellamonica Jr., S. La Fleur, V. R\"odl]{Domingos Dellamonica Jr., Steven La Fleur, Vojt\v{e}ch R\"odl}  \address{Department of Mathematics and Computer Science \\
Emory University\\ Atlanta\\ USA} \email{\{ddellam, slafleu, rodl\}@emory.edu} \thanks{The fourth author was
partially supported by NSF grants DMS-1102086 and DMS-1301698}

\author[M. Schacht]{Mathias Schacht} \address{Fachbereich Mathematik\\ Universit\"at Hamburg\\ Hamburg\\ Germany} \email{schacht@math.uni-hamburg.de} \thanks{The
fifth author was supported through the Heisenberg-Programme of the DFG\@.}

\dedicatory{Dedicated to the memory of Jirka Matou\v sek}

\begin{abstract}
The induced Ramsey number $\irnumber{F}$ of a $k$-uniform hypergraph $F$ is the smallest natural number $n$ for which there exists a $k$-uniform hypergraph $G$ on $n$ vertices such that every two-coloring of the edges of $G$ contains an induced monochromatic copy of~$F$. We study this function, showing that $\irnumber{F}$ is bounded above by a reasonable power of $r(F)$. In particular, our result implies that $\irnumber{F} \leq 2^{2^{ct}}$ for any $3$-uniform hypergraph $F$ with $t$ vertices, mirroring the best known bound for the usual Ramsey number. The proof relies on an application of the hypergraph container method.
\end{abstract}

\maketitle

\section{Introduction}

The {\it Ramsey number} $r(F; q)$ of a $k$-uniform hypergraph $F$ is the smallest natural number~$n$ such that every $q$-coloring of the edges of $K_n^{(k)}$, the complete $k$-uniform hypergraph on~$n$ vertices, contains a monochromatic copy of $F$. In the particular case when $q = 2$, we simply write $r(F)$. The existence of $r(F;q)$ was established by Ramsey in his foundational paper~\cite{R30} and there is now a large body of work studying the Ramsey numbers of graphs and hypergraphs. For a recent survey, we refer the interested reader to~\cite{CFS15}.

In this paper, we will be concerned with a well-known refinement of Ramsey's theorem, the induced Ramsey theorem. We say that a $k$-uniform hypergraph $F$  is an {\it induced subgraph} of another $k$-uniform hypergraph $G$ if $V(F) \subset V(G)$ and any $k$ vertices in $F$ form an edge if and only if they also form an edge in $G$. The {\it induced Ramsey number} $\irnumber{F; q}$ of a $k$-uniform hypergraph $F$ is then the smallest natural number $n$ for which there exists a $k$-uniform hypergraph $G$ on $n$ vertices such that that every $q$-coloring of the edges of $G$ contains an induced monochromatic copy of $F$. Again, in the particular case when $q = 2$, we simply write $\irnumber{F}$.

For graphs, the existence of induced Ramsey numbers was established
independently by Deuber~\cite{D75}, Erd\H{o}s, Hajnal, and
P\'osa~\cite{EHP75}, and R\"{o}dl~\cite{R73}, while for $k$-uniform
hypergraphs with $k \geq 3$ their existence was shown independently by
Ne\v set\v ril and R\"odl~\cite{NR77} and Abramson and
Harrington~\cite{AH78}.  The bounds that these original proofs gave on
$\irnumber{F; q}$ were enormous. However, at that time it was noted by
R\"odl (unpublished) that for bipartite graphs $F$ the induced Ramsey
numbers are exponential in the number of vertices. Moreover, it was
conjectured by Erd\H{o}s~\cite{E84} that there exists a constant $c$
such that every graph $F$ with $t$ vertices satisfies
$\irnumber{F} \leq 2^{ct}$. If true, the complete graph would show
that this is best possible up to the constant $c$. A result of Conlon,
Fox, and Sudakov~\cite{CFS12}, building on earlier work by Kohayakawa,
Pr\"omel, and R\"odl~\cite{KPR98}, comes close to establishing this
conjecture, showing that
\[\irnumber{F} \leq 2^{ct \log t}.\]
However, the method used to prove this estimate only works in the $2$-color case. For $q \geq 3$, the best known bound, due to Fox and Sudakov~\cite{FS09}, is $\irnumber{F; q} \leq 2^{c t^3}$, where $c$ depends only on $q$. 

In this note, we study the analogous question for hypergraphs, showing that the induced Ramsey number is never significantly larger than the usual Ramsey number. Our main result is the following.

\begin{theorem}\label{thm:main}
  Let $F$ be a $k$-uniform hypergraph with $t$ vertices and $\ell$ edges.
  Then there are positive constants $c_1, c_2,$ and $c_3$ such that
  \[
  \irnumber{F; q} \leq
  2^{c_1 k \ell^3\log(q t \ell)}R^{c_2 k \ell^2 + c_3t\ell},
  \]
  where $R = \nirnumber{F; q}$ is the classical $q$-color Ramsey number of $F$.
\end{theorem}

Define the tower function $t_k(x)$ by $t_1(x) = x$ and, for $i \geq 1$, $t_{i+1}(x) = 2^{t_i(x)}$. A seminal result of Erd\H{o}s and Rado~\cite{ER52}
says that
\[r(K_t^{(k)}; q) \leq t_k(c t),\]
where $c$ depends only on $k$ and $q$. This yields the following immediate corollary of Theorem~\ref{thm:main}.

\begin{corollary} \label{cor:main}
For any natural numbers $k \geq 3$ and $q \geq 2$, there exists a constant $c$ such that if $F$ is a $k$-uniform hypergraph with $t$ vertices, then
\[\irnumber{F; q} \leq t_k(c t).\]
\end{corollary}

A result of Erd\H{o}s and Hajnal (see, for example, Chapter 4.7
in~\cite{GRS90} and~\cite{CFS13}) says that
\[r(K_t^{(k)}; 4) \geq t_k(c' t),\]
where $c'$ depends only on $k$. Therefore, the Erd\H{o}s--Rado bound is sharp up to the constant~$c$ for $q \geq 4$. By taking $F = K_t^{(k)}$, this also implies that Corollary~\ref{cor:main} is tight up to the constant~$c$ for $q \geq 4$. Whether it is also sharp for $q = 2$ and $3$ depends on whether $r(K_t^{(k)}) \geq t_k(c' t)$, though determining if this is the case is a famous, and seemingly difficult, open problem.

The proof of Theorem~\ref{thm:main} relies on an application of the hypergraph container method of Saxton and Thomason~\cite{ST15} and Balogh, Morris, and Samotij~\cite{BMS15}. In Ramsey theory, the use of this method was pioneered by Nenadov and Steger~\cite{NS15} and developed further by R\"odl, Ruci\'nski, and Schacht~\cite{RRS15} in order to give an exponential-type upper bound for Folkman numbers. Our modest results are simply another manifestation of the power of this beautiful method.

\section{Proof of Theorem~\ref{thm:main}}

In order to state the result we first need some definitions. Recall that the degree $d(\sigma)$ of a set of vertices $\sigma$ in a hypergraph $H$ is the number of edges of $H$ containing $\sigma$, while the average degree is the average of $d(v) := d(\{v\})$ over all vertices $v$.

\begin{definition} \label{def:codeg} Let $H$ be an $\ell$-uniform hypergraph of order
  $N$ with average degree~$d$.  Let $\tau > 0$.  Given $v \in V(H)$ and
  $2 \leq j \leq \ell$, let
  \begin{equation*}\label{eq:def_deg_j}
    d^{(j)}(v) = \max\bigl\{ d(\sigma) \: : \: v \in \sigma \subset V(H),
    |\sigma|=j \bigr\}.
  \end{equation*}
  If $d > 0$, define $\delta_j$ by the equation
  \begin{equation*}\label{eq:def_delta_j}
    \delta_j \tau^{j-1} N d = \sum_{v} d^{(j)}(v).
  \end{equation*}
  The \emph{codegree function} $\delta(H,\tau)$ is then defined by
  \begin{equation*}\label{eq:def_codeg}
    \delta(H,\tau) = 2^{\binom{\ell}{2}-1}\sum_{j=2}^{\ell}
    2^{-\binom{j-1}{2}}\delta_j.
  \end{equation*}
  If $d=0$, define $\delta(H,\tau) = 0$.
\end{definition}

The precise lemma we will need is a slight variant of Corollary 3.6 from Saxton and Thomason's paper~\cite{ST15}. A similar version was already used in the work of R\"odl, Ruci\'nski, and Schacht~\cite{RRS15} and we refer the interested reader to that paper for a thorough discussion.

\begin{lemma}\label{lm:ST}
  Let $H$ be an $\ell$-uniform hypergraph on $N$ vertices  with average degree~$d$.
  Let $0 < \eps < 1/2$.  Suppose that $\tau$ satisfies $\delta(H,\tau)
  \leq \eps/12\ell!$ and $\tau \leq 1/144\ell!^2\ell$.  Then there
  exists a collection $\cC$ of subsets of $V(H)$
  such that
  \begin{enumerate}[label=\rmlabel]
  \item\label{item_almost_independent} for every set $I \subset V(H)$ such that 
    $e(H[I]) \leq \eps \tau^\ell e(H)$, there is $C
    \in \cC$ with $I \subset C$,
  \item\label{item_few_edges_in_C} $e(H[C]) \leq \eps e(H)$ for all $C \in
    \cC$,
  \item\label{item_size_cC} $\log|\cC| \leq 1000 \ell!^3 \ell \log(1/\eps) N \tau \log(1/\tau)$.
  \end{enumerate}
\end{lemma}

Before we give the proof of Theorem~\ref{thm:main}, we first describe the
$\ell$-uniform hypergraph $H$ to which we will apply Lemma~\ref{lm:ST}.

\begin{construction}\label{cnst:H}
Given a $k$-uniform hypergraph $F$ with $\ell$ edges, we construct an auxiliary hypergraph $H$ by taking
\[
	 V(H) = \binom{[n]}{k}
	 \qqand
	 E(H) = \Bigg\{ E \in \binom{V(H)}{\ell} \colon
   	 E \cong F \Bigg\}.
\]
  In other words, the vertices of $H$ are the $k$-tuples of $[n]$ and
  the edges of $H$ are copies of~$F$ in $\binom{[n]}{k}$.
\end{construction}

\begin{proof}[Proof of Theorem~\ref{thm:main}]
  Recall that $R = \nirnumber{F; q}$, the $q$-color Ramsey number of
  $F$, and suppose that $F$ has $t$ vertices and $\ell$ edges.  Let us
  fix the following numbers:
  \begin{equation}\label{eq:values}
    \begin{split}
      \tau = n^{-\frac{1}{2\ell}}, &\qquad p=1000R^k q \alpha, \qquad
      \alpha=n^{-\frac{1}{2\ell} + \frac{1}{4\ell(\ell+1)}}, \\
      \eps = 1/(2qR^t), &\qquad n =
      \ell^{40\ell^2(\ell+1)}(1000q)^{8\ell(\ell+1)}R^{4k\ell(\ell+1)+4t\ell}\binom{t}{k}^{4\ell}.
    \end{split}
  \end{equation}
  \begin{remark}
    Note that $n$ is bounded above by an expression of the form
    \[2^{c_1 k \ell^3\log(q t \ell)}R^{c_2 k \ell^2 + c_3t\ell},\]
    as required.
  \end{remark}

  Obviously, $R\geq t$ and one can check that $p$ and $n$ satisfy the
  following conditions, which we will make use of during the course of
  the proof:
  \begin{align}
    \label{eq:p}
       p &\leq 1,\\
      \label{eq:n1}
      n &\geq (24 \cdot 2^{\binom{\ell}{2}} t^t q \ell! R^t)^2, \\
      \label{eq:n2}
      n &> (144 \ell!^2 \ell)^{2\ell},\\
      \label{eq:n3}
      n &> \ell^{40\ell^2(\ell+1)},\\
      \label{eq:n4}
      n &> (1000 q)^{8\ell(\ell+1)}R^{4k\ell(\ell+1) + 4t\ell}
      \binom{t}{k}^{4\ell}.
  \end{align}

  We will show that, with positive probability, a random hypergraph
  $G \in \gnp$ has the property that every $q$-coloring of its edges
  contains an induced monochromatic copy of~$F$.  The proof proceeds
  in two stages.  First, we use Lemma~\ref{lm:ST} to show that, with
  probability~$1-o(1)$, $G$ has the property that any $q$-coloring of
  its edges yields many monochromatic copies of $F$.  Then we show that
  some of these monochromatic copies must be induced.

  More formally, let $X$ be the event that there is a $q$-coloring of
  the edges of $G$ which contains at most
  \begin{equation*}\label{eq:def_mono}
    \monosize := \frac{\eps \tau^\ell (n)_t}{\aut(F)}
  \end{equation*}
  monochromatic copies of $F$ in each color, and let $Y$ be the event
  that $G$ contains at least $\monosize$ noninduced
  copies of $F$. Note that if $\overline{X} \cap \overline{Y}$ happens, then, in any
  $q$-coloring, there are more monochromatic copies of $F$ in one of the
  $q$ colors than there are noninduced copies of $F$ in $G$.  Hence,
  that color class must contain an induced copy of $F$.

  We now proceed to show that the probability $\PR(X)$ tends to zero as $n$ tends to infinity.  In
  order to apply Lemma~\ref{lm:ST}, we need to check that $\tau$ and
  $\eps$ satisfy the requisite assumptions with respect to the $\ell$-uniform hypergraph $H$
  defined in Construction~\ref{cnst:H}.  Let~$\sigma
  \subset V(H)$ be arbitrary and define
  \[
  V_\sigma = \bigcup_{v \in \sigma} v \subset [n].
  \]
  For an arbitrary set~$W \subset [n] \setminus V_\sigma$ with
  $|W| = t - |V_\sigma|$, let $\emb_F(\sigma, W)$ denote the number of
  copies $\widetilde{F}$ of $F$ with
  $V(\widetilde{F}) = W \cup V_\sigma$ and
  $\sigma \subset E(\widetilde{F})$. Observe that this number does not
  actually depend on the choice of~$W$, so we will simply use
  $\emb_F(\sigma)$ from now on.

  Since there are clearly $\binom{n-|V_\sigma|}{t-|V_\sigma|}$ choices
  for the set $W$, we arrive at the following claim.

  \begin{claim}\label{claim:calculation_dj}
    For any~$\sigma \subset V(H)$,
    \[
    d(\sigma) = \binom{n - |V_\sigma|}{t - |V_\sigma|}
    \emb_F(\sigma).\qedhere
    \]
  \end{claim}

  Let us denote by $t_j$ the minimum number of vertices of $F$ which
  span $j$ edges.  From Claim~\ref{claim:calculation_dj}, it follows
  that for any~$\sigma \subset V(H)$ with $|\sigma| = j$, we have
  \[
  d(\sigma) = \binom{n - |V_\sigma|}{t - |V_\sigma|} \emb_F(\sigma)
  \leq \binom{n - t_j}{t - t_j} \emb_F(\sigma).
  \]
  On the other hand, for a singleton~$\sigma_1 \subset V(H)$, we have
  $|V_{\sigma_1}| = k$ and therefore $d = d(\sigma_1)$ is such that
  \[
  \frac{d(\sigma)}{d} \leq \frac{\binom{n - t_j}{t -
      t_j}}{\binom{n-k}{t-k}} \frac{\emb_F(\sigma)}{\emb_F(\sigma_1)}
  \leq \frac{\binom{n - t_j}{t - t_j}}{\binom{n-k}{t-k}} <
  \left(\frac{n}{t}\right)^{k - t_j}.
  \]
  It then follows from Definition~\ref{def:codeg} and \eqref{eq:values}
  that
  \begin{equation}\label{eq:delta-j}
    \delta_j <
    \frac{(n/t)^{k-t_j}}{\tau^{j-1}} < t^{t} n^{k-t_j+(j-1)/(2\ell)}.
  \end{equation}
  Since $t_j$ is increasing with respect to $j$, $t_2
  \geq k+1$, and $j \leq \ell$, we have $k - t_j + \frac{j-1}{2\ell} \leq
  -1/2$. Thus, in view of~\eqref{eq:delta-j}, we have
 \begin{equation}\label{eq:gamma}
    \delta_j < t^t n^{k-t_j+(j-1)/(2\ell)} \leq t^t n^{-1/2}
  \end{equation}
  for all $2 \leq j \leq \ell$.

  Using Definition~\ref{def:codeg} and inequality~\eqref{eq:gamma}, we can 
  now bound the codegree function $\delta(H,\tau)$ by
  \begin{equation}\label{eq:bound_delta}
    \begin{split}
      \delta(H,\tau) = 2^{\binom{\ell}{2}-1} \sum_{j=2}^{\ell}
      2^{-\binom{j-1}{2}} \delta_j
      \leq 2^{\binom{\ell}{2}-1} t^t n^{-1/2} \sum_{j=2}^{\ell} 2^{-\binom{j-1}{2}}
      \leq 2^{\binom{\ell}{2}} t^t n^{-1/2}.
    \end{split}
  \end{equation}
  Since $n$ satisfies~\eqref{eq:n1},
  inequality~\eqref{eq:bound_delta} implies that
  \[
  \delta(H,\tau) \leq 2^{\binom{\ell}{2}} t^t n^{-1/2} \leq
  \frac{\eps}{12\ell!}.
  \]
  That is, $\delta(H,\tau)$ satisfies the condition in Lemma~\ref{lm:ST}.

  Finally,~\eqref{eq:n2} implies that $\tau$ satisfies the condition
  \[
  \tau = n^{-1/(2\ell)} < \frac{1}{144\ell!^2\ell}.
  \]
  Therefore, the assumptions of Lemma~\ref{lm:ST} are met and we can let $\cC$
  be the collection of subsets from $V(H)$ obtained from applying
  Lemma~\ref{lm:ST}. Denote the elements of $\cC$ by $C_1, C_2, \dots,
  C_{|\cC|}$.

  For every choice of $1\leq a_1, \dots, a_q \leq |\cC|$ (not necessarily distinct) let $E_{a_1, \dots, a_q}$
  be the event that $G\subseteq C_{a_1} \cup \dots \cup C_{a_q}$.
  Next we will show the following claim.
  
  \begin{claim}\label{claim:E_ab}
    \begin{equation}\label{eq:union_bound_X}
      \PR(X) \leq \PR\biggl( \bigvee_{a_1, \dots, a_q} E_{a_1, \dots, a_q} \biggr) \leq \sum_{a_1, \dots, a_q} \PR(E_{a_1,\dots, a_q}).
    \end{equation}
  \end{claim}
  
  \begin{proof}
    Suppose that $G\in X$. By definition, there
    exists a $q$-coloring of the edges of $G$, say with colors $1, 2, \dots, q$, which contains at
    most~$\monosize$ copies of~$F$ in each color. For any color 
    class $j$, let $I_j$ denote the set of vertices of $H$ which 
    correspond to edges of color $j$ in $G$. Since each edge in 
    $H[I_{j}]$ corresponds to a copy of $F$ in color $j$, we 
    have $e(H[I_{j}]) \leq \monosize$. Note that
    \[
    \monosize = \eps \tau^\ell e(H),
    \]
    which means that each~$I_{j}$ satisfies the condition~\eqref{item_almost_independent} of
    Lemma~\ref{lm:ST}. Therefore, for each color class $j$,
    there must be a set $C_{a_j} \in \cC$ such that $C_{a_j} \supset I_{j}$. 
    Since $G = \bigcup_j I_j$, this implies that $G\in E_{a_1, \dots, a_q}$. 
    Since $G\in X$ was arbitrary, the bound
    \eqref{eq:union_bound_X} follows and the claim is proved.
  \end{proof}

  Owing to Claim~\ref{claim:E_ab}, we now bound $\PR(E_{a_1, \dots, a_q})$.
  Recalling the definition of the event $E_{a_1, \dots, a_q}$, we note that
  \begin{equation}\label{eq:prob_Eab}
    \PR(E_{a_1, \dots, a_q}) = (1-p)^{|V(H) \setminus (C_{a_1} \cup \dots \cup C_{a_q})|}.
  \end{equation}
  Hence, we shall estimate $|V(H) \setminus (C_{a_1} \cup \dots \cup C_{a_q})|$
  to derive a bound for $\PR(X)$ by~\eqref{eq:union_bound_X}.

  \begin{claim}\label{claim:complement_is_large}
    For all choices $1\leq a_1, \dots, a_q \leq |\cC|$ we have
    \begin{equation*}\label{eq:complement_is_large}
      |V(H) \setminus (C_{a_1} \cup \dots \cup C_{a_q})| \geq \frac{1}{2}\Bigl( \frac{n}{R} \Bigr)^k.
    \end{equation*}
  \end{claim}
  
  \begin{proof} Let $a_1, \dots, a_q$ be fixed and set
    \begin{equation}\label{eq:family_A}
      \cA = \biggl\{ A \in \binom{[n]}{R} \: : \: \binom{A}{k} \subset C_{a_1} \cup \dots \cup C_{a_q} \biggr\}.
    \end{equation}
    By the definition of~$R = \nirnumber{F;q}$, for each
    set~$A \in \cA$ there is an index $j = j(A) \in [q]$ such that
    $C_{a_j}$ contains a copy of~$F$ with vertices from~$A$. The
    element $e \in E(C_{a_j})$ that corresponds to
    this copy of~$F$ satisfies $e \subset \binom{A}{k}$ and,
    thus, $\bigcup_{x \in e} x \subset A$. We now give an upper
    bound for $|\cA|$ by counting the number of pairs in
    \[
    \cP = \biggl\{ (e, A) \in \bigcup_{i = 1}^q E\bigl(C_{a_i}\bigr)
    \times \cA \text{ with } \bigcup_{x \in e} x \subset A \biggr\}.
    \]
    On the one hand, we have already established that $|\cP| \geq |\cA|$. On
    the other hand, for any fixed $e \in E(H)$, we have
    $|\bigcup_{x \in e} x| = |V(F)| = t$ and, therefore, there are at
    most~$\binom{n - t}{R - t}$ sets~$A \supset \bigcup_{x \in e}
    x$. It follows that
    \begin{equation}\label{eq:pair_count}
      \begin{split}
        |\cA| \leq |\cP| & \leq \biggl|\bigcup_{i = 1}^q
        E\bigl(C_{a_i}\bigr)\biggr| \binom{n - t}{R - t}
        \overset{\eqref{item_few_edges_in_C}}{\leq} q \eps e(H)
        \binom{n - t}{R - t}\\ & \overset{\eqref{eq:values}}{=}
        \frac{e(H)}{2R^t} \binom{n - t}{R - t} \leq \frac{(n)_t}{2R^t}
        \binom{n - t}{R - t} \leq \frac{1}{2} \binom{n}{R}.
      \end{split}
    \end{equation}
    By definition, each~$A \in \binom{[n]}{R} \setminus \cA$
    satisfies~$\binom{A}{k} \nsubset C_{a_1}\cup\dots\cup C_{a_q}$.
    Hence, $V(H)\setminus (C_{a_1}\cup\dots\cup C_{a_q})$
    intersects~$\binom{A}{k}$.  Since an element of~$V(H)$ can appear
    in at most~$\binom{n - k}{R - k}$ sets~$A$, it follows
    from~\eqref{eq:pair_count} that there are at least
    \[
    \frac{1}{2} \binom{n}{R} \biggl / \binom{n-k}{R-k} \geq
    \frac{1}{2}\Bigl( \frac{n}{R} \Bigr)^k
    \]
    elements in $V(H) \setminus (C_{a_1} \cup \dots \cup C_{a_q})$, as required. 
  \end{proof}
  
  In view of Claim~\ref{claim:complement_is_large}, our choice of~$p =
  1000 R^k q \alpha$, where $\alpha = n^{-1/2\ell + 1/4\ell(\ell+1)}$,
  and~\eqref{eq:prob_Eab}, we have, for any~$C_{a_1}, \dots, C_{a_q} \in \cC$,
  \begin{equation}\label{eq:bound-on-p}
  \begin{split}
    \PR(E_{a_1, \dots, a_q}) &\leq (1-p)^{(n/R)^k/2}\\
    &\leq \exp\bigl(-pn^k/2R^k\bigr) = \exp\bigl(-(1000R^k q \alpha)n^k/2R^k\bigr) \\
    &= e^{-500 q \alpha n^k} \leq e^{-1000 q \alpha N},
  \end{split}
  \end{equation}
  where, in the last step, we used $N = \binom{n}{k} \leq \frac{n^k}{2}$.
  Therefore,~\eqref{eq:union_bound_X} and~\eqref{eq:bound-on-p}
  together with the bound on~$|\cC|$ given by
  Lemma~\ref{lm:ST}(\ref{item_size_cC}) imply that
  \[
  \begin{split}
    \PR(X) &\leq \sum_{C_{a_1}, \dots, C_{a_q} \in \cC} \PR(E_{a_1, \dots, a_q})\ 
    \leq |\cC|^q e^{-1000 q \alpha N}\\
    &\leq \exp\bigl(1000 q \ell!^3 \ell \log(1/\eps) N \tau \log(1/\tau)
    -1000 q \alpha N \bigr) \\
    &= \exp\bigl(1000 q N\tau (\ell!^3 \ell \log(1/\eps) \log(1/\tau) - \alpha/\tau)
    \bigr)\\
    &\leq\exp\bigl( 1000 q N\tau (\ell!^3 \log^2{n} - n^{1/(4\ell(\ell+1))})
    \bigr) \leq 1/4,\\
  \end{split}
  \]
  where we used that $n$ satisfies~\eqref{eq:n3}.

  Now, by Markov's inequality, with probability at least $1/2$, the
  number of noninduced copies of~$F$ in $G$ will be at most twice the
  expected number of copies, which is fewer than
  \[
  \begin{split}
    2 p^{\ell+1}\frac{(n)_t}{\aut(F)}\binom{t}{k}
    &= 2(1000 q)^{\ell+1}R^{k(\ell+1)}n^{-1/2-1/(4\ell)}\frac{(n)_t}{\aut(F)}\binom{t}{k}\\
    &< \frac{1}{2qR^t}(n^{-1/(2\ell)})^\ell \frac{(n)_t}{\aut(F)} =
    \eps \tau^\ell \frac{(n)_t}{\aut(F)} = M,
  \end{split}
  \]
  where the inequality above follows from~\eqref{eq:n4}. In other words, $\PR(\overline{Y})\geq 1/2$ and, therefore, $\PR(\overline{X} \cap \overline{Y})
  \geq 1/4$, so there exists a graph $G$ such that $\overline{X} \cap \overline{Y}$ holds. By our
  earlier observations, this completes the proof.
  \end{proof}

\section{Concluding remarks}

Beginning with Fox and Sudakov~\cite{FS08}, much of the recent work on induced Ramsey numbers for graphs has used pseudorandom rather than random graphs for the target graph $G$. The results of this paper rely very firmly on using random hypergraphs. It would be interesting to know whether comparable bounds could be proved using pseudorandom hypergraphs.

It would also be interesting to prove comparable bounds for the following variant of the induced Ramsey theorem, first proved by Ne\v set\v ril and R\"odl~\cite{NR75}: for every graph $F$, there exists a graph $G$ such that every $q$-coloring of the triangles of $G$ contains an induced copy of $F$ all of whose triangles receive the same color. By taking $F = K_t$ and $q = 4$, we see that $|G|$ may need to be double exponential in $|F|$. We believe that a matching double-exponential upper bound should also hold.

\end{document}